\documentclass{amsart}

\numberwithin{equation}{section}

\usepackage[utf8]{inputenc}
\usepackage[T1]{fontenc}
\usepackage{lmodern}

\usepackage[margin=1.3in]{geometry}
\usepackage{setspace}
\usepackage[usenames, dvipsnames]{xcolor}
\usepackage{graphicx}
\usepackage{mathtools}
\usepackage{amssymb}
\usepackage{amsthm}
\usepackage{xparse}
\usepackage{xspace}
\usepackage{enumitem}
\usepackage[backref=page]{hyperref}
\usepackage[capitalize, noabbrev]{cleveref}

\DeclareDocumentCommand{\shortexact}{s O{} O{} mmmm}{
\IfBooleanTF{#1}{ 
\begin{tikzcd}[ampersand replacement=\&]
        {1} \& {#4} \& {#5} \& {#6} \& {1#7}
        \arrow[from=1-1, to=1-2]
        \arrow["#2", from=1-2, to=1-3]
        \arrow["#3", from=1-3, to=1-4]
        \arrow[from=1-4, to=1-5]
\end{tikzcd}
}{ 
\begin{tikzcd}[ampersand replacement=\&]
        {0} \& {#4} \& {#5} \& {#6} \& {0#7}
        \arrow[from=1-1, to=1-2]
        \arrow["#2", from=1-2, to=1-3]
        \arrow["#3", from=1-3, to=1-4]
        \arrow[from=1-4, to=1-5]
\end{tikzcd}
}}

\newcommand{\Q}{\mathbb Q}
\newcommand{\R}{\mathbb R}
\newcommand{\C}{\mathbb C}

\newcommand{\Z}{\mathbb Z}

\newcommand{\RP}{\mathbb{RP}}

\newcommand{\HP}{\mathbb{HP}}

\renewcommand{\O}{\mathrm{O}}
\newcommand{\SO}{\mathrm{SO}}
\newcommand{\Spin}{\mathrm{Spin}}
\newcommand{\Pin}{\mathrm{Pin}}
\newcommand{\U}{\mathrm{U}}
\newcommand{\SU}{\mathrm{SU}}

\newcommand{\Spinc}{\relax\ifmmode{\Spin^c}\else Spin\textsuperscript{$c$}\xspace\fi}
\newcommand{\Spinh}{\relax\ifmmode{\Spin^h}\else Spin\textsuperscript{$h$}\xspace\fi}
\newcommand{\spinc}{spin\textsuperscript{$c$}\xspace}
\newcommand{\Pinc}{\relax\ifmmode{\Pin^c}\else Pin\textsuperscript{$c$}\xspace\fi}
\newcommand{\pinc}{pin\textsuperscript{$c$}\xspace}
\newcommand{\Pinp}{\relax\ifmmode{\Pin^+}\else Pin\textsuperscript{$+$}\xspace\fi}
\newcommand{\pinp}{pin\textsuperscript{$+$}\xspace}
\newcommand{\Pinm}{\relax\ifmmode{\Pin^-}\else Pin\textsuperscript{$-$}\xspace\fi}
\newcommand{\pinm}{pin\textsuperscript{$-$}\xspace}
\newcommand{\spinh}{spin\textsuperscript{$h$}\xspace}

\newcommand{\pinhm}{pin\textsuperscript{$h-$}\xspace}
\newcommand{\sm}{\mathrm{sm}}

\newtheorem{thm}[equation]{Theorem}
\newtheorem{lem}[equation]{Lemma}
\newtheorem{cor}[equation]{Corollary}

\newtheorem*{main_thm_ref}{\cref{main_thm}}

\theoremstyle{definition}

\newtheorem{exm}[equation]{Example}
\newtheorem{defn}[equation]{Definition}
\newtheorem{constr}[equation]{Construction}

\newtheorem{ques}[equation]{Question}

\theoremstyle{remark}
\newtheorem{rem}[equation]{Remark}

\crefname{thm}{Theorem}{Theorems}
\crefname{lem}{Lemma}{Lemmas}
\crefname{cor}{Corollary}{Corollaries}
\crefname{prop}{Proposition}{Propositions}
\crefname{ex}{Exercise}{Exercises}
\crefname{exm}{Example}{Examples}
\crefname{defn}{Definition}{Definitions}
\crefname{constr}{Construction}{Constructions}
\crefname{claim}{Claim}{Claims}
\crefname{rem}{Remark}{Remarks}
\crefname{fct}{Fact}{Facts}
\crefname{note}{Note}{Notes}

\newcommand{\term}{\emph}

\newcommand{\vp}{\varphi}

\newcommand{\inj}{\hookrightarrow}

\newcommand{\pt}{\mathrm{pt}}

\let\shortmapsto\mapsto
\renewcommand{\mapsto}{\mathchoice{\longmapsto}{\shortmapsto}{\shortmapsto}{\shortmapsto}}
\DeclarePairedDelimiter\paren{(}{)}
\DeclarePairedDelimiter\ang{\langle}{\rangle}
\DeclarePairedDelimiter\abs{\lvert}{\rvert}

\DeclarePairedDelimiter\bkt{[}{]}
\DeclarePairedDelimiter\set{\{}{\}}

\let\oldparen\paren
\def\paren{\@ifstar{\oldparen}{\oldparen*}}
\let\oldbkt\bkt
\def\bkt{\@ifstar{\oldbkt}{\oldbkt*}}

\newcommand{\newoperator}[1]{\expandafter\DeclareMathOperator\csname #1\endcsname{\operatorname{#1}}}

\newoperator{Aut}
\newoperator{Ext}
\newoperator{Hom}

\newcommand{\NewThomSpectrum}[1]{\expandafter\newcommand\csname M#1\endcsname{\mathit{M#1}}}
\newcommand{\NewMTSpectrum}[1]{\expandafter\newcommand\csname MT#1\endcsname{\mathit{MT#1}}}
\newcommand{\BothThomSpectra}[1]{\NewThomSpectrum{#1}\NewMTSpectrum{#1}}
\BothThomSpectra{O}
\BothThomSpectra{SO}
\BothThomSpectra{Spin}
\BothThomSpectra{Pin}
\BothThomSpectra{U}
\newcommand{\GL}{\mathrm{GL}}
\AtBeginDocument{
 \usepackage{microtype}
}
\hypersetup{
 colorlinks,
 linkcolor={red!50!black},
 citecolor={green!50!black},
 urlcolor={blue!80!black}
}

\usepackage{tikz,tikz-cd}
\usepackage[normalem]{ulem}
\title{Smith homomorphisms and \texorpdfstring{$\Spin^{\MakeLowercase{h}}$}{Spinh} structures}

\author{Arun Debray}
\address{Department of Mathematics, Purdue University, 150 N. University Street, West Lafayette, IN 47907}
\email{\href{mailto:a.debray@gmail.com}{a.debray@gmail.com}}
\urladdr{\href{https://adebray.github.io/}{https://adebray.github.io/}}

\author{Cameron Krulewski}
\address{Department of Mathematics, Massachusetts Institute of Technology,
77 Massachusetts Avenue,
\indent Cambridge, MA 02139}
\email{\href{mailto:camkru@mit.edu}{camkru@mit.edu}}
\urladdr{\href{https://cakrulewski.github.io}{https://cakrulewski.github.io}}

\date{\today}

\begin{document}

%

\begin{abstract}
%
%
In this article, we answer two questions of Buchanan-McKean~\cite{BM23} about bordism for manifolds with \spinh structures: we establish a Smith isomorphism between the reduced \spinh bordism of $\RP^\infty$ and \pinhm bordism, and we provide a geometric explanation for the isomorphism $\Omega_{4k}^{\Spin^c}\otimes\Z[1/2] \cong\Omega_{4k}^{\Spin^h}\otimes\Z[1/2]$. Our proofs use the general theory of twisted spin structures and Smith homomorphisms that we developed in~\cite{DDKLPTT24} joint with Devalapurkar, Liu, Pacheco-Tallaj, and Thorngren,
specifically that the Smith homomorphism participates in a long exact sequence with explicit, computable terms.
%
%
%
%
%
\end{abstract}

\maketitle

\tableofcontents

\section{Introduction}
At the start of the 1960s, C.T.C.\ Wall challenged the readers of~\cite{Wal60} to study the bordism groups of spin manifolds---and by the end of the decade, Anderson-Brown-Peterson~\cite{anderson_structure_1967} had completely solved this problem, determining not just the spin bordism groups but also a convenient decomposition of the spectrum $MT\Spin$ itself, catalyzing computations of other, related bordism groups.

One such example is bordism for a complex analogue of spin structures, referred to as \spinc structures (see \cref{spinc_tw_spin}), which was solved almost immediately after Anderson-Brown-Peterson's work (see~\cite[Chapter XI]{Sto68}). Similarly, one can
replace the complex numbers with the quaternions, leading to the notion of a 
\term{\spinh structure}, i.e.\ a reduction of structure group to the group\footnote{Here and elsewhere in this article, the notation $G\times_{\set{\pm 1}} H$ indicates that there are central subgroups $\set{\pm 1}\subset G$, $\set{\pm 1}\subset H$ each isomorphic to the multiplicative group $\set{\pm 1}\subset\R^\times$; then $G\times_{\set{\pm 1}} H$ is the quotient of $G\times H$ by the diagonal $\set{\pm 1}$ subgroup. These subgroups of $G$ and $H$ will be clear from context.}
\begin{equation}
    \Spin_n^h\coloneqq \Spin_n\times_{\set{\pm 1}}\SU_2.
\end{equation}

Spin$^h$ structures have been studied in the mathematics and physics literature since the 1960s, with applications to quantum gravity~\cite{BFF78, Bec24}, index theory, e.g.\ in~\cite{May65, Nag95, Bar99, FH21, Che17}, Seiberg-Witten theory~\cite{OT96}, immersion problems~\cite{Bar99, AM20}, almost quaternionic geometry, e.g.\ in~\cite{Nag95, Bar99, AM20}, and invertible field theories~\cite{FH21, BC18, WWW19, WWZ19, WW20a, DL21, DY22, Ste22, WWW22, BI23, DDKLPTT23}. See~\cite{Law23} for a review of the mathematical aspects of \spinh structures.

However, \spinh \emph{bordism} has attracted interest only in the last few years, beginning with Freed-Hopkins' work~\cite{FH21} applying low-degree \spinh bordism groups to condensed-matter physics; other important results include obstruction theory for \spinh structures~\cite{AM20}, the construction of a quaternionic Atiyah-Bott-Shapiro map~\cite{FH21, Hu23} and an Anderson-Brown-Peterson-style splitting of the \spinh bordism spectrum at the prime $2$~\cite{Mil23}.\footnote{However, not everything carries over: just as the quaternions have less structure than $\R$ or $\C$, \spinh bordism has less structure than spin or \spinc bordism. For example, $\Omega_*^\Spin$ and $\Omega_*^{\Spin^c}$ have ring structures induced from the direct product of manifolds, but $\Omega_*^{\Spin^h}$ does not. Thus the twisted Atiyah-Bott-Shapiro map mentioned above is not a ring homomorphism.} 

Recently, Buchanan-McKean~\cite{BM23} proved a number of key results on \spinh bordism, including describing the above splitting in terms of characteristic classes and showing that a collection of characteristic classes valued in quaternionic $K$-theory detect a manifold's \spinh bordism class. Using this splitting, they give an algorithm for computing $\Omega_n^{\Spin^h}$ for all $n$ and analyze the asymptotics of the size of the $n^{\mathrm{th}}$ \spinh bordism group in $n$.


Buchanan-McKean also ask several questions on \spinh bordism~\cite[\S 10]{BM23} coming from their work. The main goal of this article is to answer two of these questions, which we now describe.

Anderson-Brown-Peterson~\cite{ABP69} established a \term{Smith isomorphism} $\sm_\sigma\colon \widetilde\Omega_n^\Spin(\RP^\infty)\overset\cong\to \Omega_{n-1}^{\Pin^-}$,\footnote{\Pinm structures are an unoriented generalization of spin structures that we discuss in \cref{pin_exm}.} described concretely by 
taking a spin manifold $M$ with a map $M\to \RP^\infty$ to the zero set of a generic section of the pullback of the tautological line bundle to $M$.
Then, Bahri-Gilkey~\cite{BG87a, BG87} constructed a completely analogous isomorphism $\sm_\sigma^c\colon \widetilde \Omega_n^{\Spin^c}(\RP^\infty)\xrightarrow{\cong} \Omega_{n-1}^{\Pin^c}$.
\begin{ques}[{Buchanan-McKean~\cite[Question 10.8]{BM23}}]
\label{spinh_smith_qn}
Let $\Pin_n^{h-}\coloneqq \Pin_n^-\times_{\set{\pm 1}} \SU_2$.\footnote{This group was first defined by Freed-Hopkins~\cite[(9.21)]{FH21}, who call it $G^-$.}
Is there a Smith isomorphism for \pinhm bordism?
\end{ques}
We affirmatively answer this question.
\begin{main_thm_ref}
    For all $n$, there is an isomorphism
    \begin{equation} \mathrm{sm}_\sigma^h\colon\widetilde\Omega_{n}^{\Spin^h}(\RP^\infty) \overset\cong\longrightarrow \Omega_{n-1}^{\Pin^{h-}}\end{equation}
    given by sending a pair $(M,f)$ of a \spinh manifold $M$ with a generic map $f\colon M\to\RP^\infty$
to the zero set of a generic section of the pullback of the tautological line bundle $\sigma\to\RP^\infty$ by $f$.
\end{main_thm_ref}
Part of this theorem is the assertion that such a zero set is generically a closed $(n-1)$-manifold with \pinhm structure.

The technique we use to prove \cref{main_thm} also enables us to solve another one of Buchanan-McKean's questions.
\begin{ques}[{Buchanan-McKean~\cite[Question 10.3]{BM23}}]\label{BM_2nd_qn}
For all $k\ge 0$, $\mathrm{rank}(\Omega_{4k}^{\Spin^c}) = \mathrm{rank}(\Omega_{4k}^{\Spin^h})$. Is there a geometric explanation for this fact? Is there a procedure to produce generators for the free summand of $\Omega_{4k}^{\Spin^h}$ from those of $\Omega_{4k}^{\Spin^c}$?
\end{ques}
To answer this question, we exhibit a map $p_*\colon\Omega_{n}^{\Spin^c}\to\Omega_{n}^{\Spin^h}$ induced from an inclusion $\Spin_n^c\hookrightarrow\Spin_n^h$. We show that $p_*$ is part of a long exact sequence of bordism groups whose third term is $\Omega_{n-3}^\Spin(B\SO_3)$~\eqref{spinc_h_smith}, and give geometric interpretations to the three maps of the long exact sequence in~\ref{LES_map_p}--\ref{final_LES}. Exactness yields a quick proof of the following theorem.

%
%

%
%
%

\newtheorem*{generator_cor_ref}{\textbf{\emph{\cref{generator_cor}}}}
\begin{generator_cor_ref}
    For all $k\ge 0$, the map
\begin{equation}
    p_*\colon \Omega_{4k}^\Spinc\otimes\Z[1/2] \longrightarrow \Omega_{4k}^\Spinh\otimes\Z[1/2],
\end{equation}
where $p_*$ is as above, is an isomorphism.
\end{generator_cor_ref}
This answers the first part of \cref{BM_2nd_qn}. Unfortunately, there is quite a bit of $2$-torsion in $\Omega_*^\Spin(B\SO_3)$, preventing us from lifting to $\Z$. This also suggests that answering the second part of Buchanan-McKean's question, building manifold generators of free summands of \spinh bordism from manifold generators of free summands of \spinc bordism, would be very difficult.

We use the same technique to prove both \cref{main_thm,generator_cor}: a method of easily producing geometrically-defined long exact sequences of bordism groups. The input is a virtual vector bundle $V$ and a vector bundle $W$ of ranks $r_V$, resp.\ $r_W$, both over a space $X$. From this data, there is a long exact sequence
\begin{equation}\label{introSmithLES}
    \dotsb \longrightarrow
    \Omega_n^{\Spin}(S(W)^{p^*V}) \overset{p_*}{\longrightarrow}
    \Omega_n^\Spin(X^{V-r_V}) \xrightarrow{\sm_W}
    \Omega_{n-r_W}^\Spin(X^{V+W-r_V-r_W})\longrightarrow
    \dotsb
\end{equation}
where $p$ denotes the bundle map $S(W)\to X$ for the sphere bundle of $W$ and $\sm_W$ is the \term{Smith homomorphism}, the map on bordism defined by taking a smooth representative of the Poincaré dual of the cobordism Euler class of $W$. This long exact sequence is natural in the data of $X$, $V$, and $W$. The spin bordism of the Thom spectrum $X^{V-r_V}$ may be interpreted in terms of \term{twisted spin bordism}: the bordism of manifolds $M$ equipped with a map $f\colon M\to X$ and a spin structure on $TM\oplus f^*(V)$ (see \cref{what_is_twisted_spin,shearing}). The exact sequence~\eqref{introSmithLES} is attributed to James and is well-known; its relationship to the Smith homomorphism is explained in our work~\cite{DDKLPTT23,DDKLPTT24} joint with Devalapurkar, Liu, Pacheco-Tallaj, and Thorngren. We call~\eqref{introSmithLES} the \term{Smith long exact sequence}. We prove \cref{main_thm,generator_cor} by making judicious choices for $X$, $V$, and $W$, then invoking exactness of the resulting instances of~\eqref{introSmithLES}. 

In \S\ref{s:bg}, we go over the background we need to prove \cref{main_thm,generator_cor}: twisted spin structures in \S\ref{s:twspin} and the Smith long exact sequence in \S\ref{s:LES}, including several examples of each. In \S\ref{smith_iso_section}, we prove \cref{main_thm}, and in \S\ref{Question10.3section}, we prove \cref{generator_cor}.



\subsection*{Acknowledgements}
We especially want to thank Jonathan Buchanan and Stephen McKean for asking the questions that inspired our project in~\cite{BM23} and for their interest in our work. In addition, we warmly thank
Yu Leon Liu,
Natalia Pacheco-Tallaj,
and
Ryan Thorngren
for conversations helpful to this paper. Part of this project was completed while AD and CK visited the Perimeter Institute for Theoretical Physics for the conference ``Higher Categorical Tools for Quantum Phases of Matter''; research at Perimeter is supported by the Government of Canada through Industry Canada and by the Province of Ontario through the Ministry of Research \& Innovation.
CK is supported by NSF DGE-2141064.


\section{Background}\label{s:bg}
Here we review the Smith long exact sequence and the concepts needed to set it up.


\subsection{Twisted spin structures}\label{s:twspin}
Recall that a \term{spin structure} on a vector bundle $W\to Y$ is defined to be a homotopy class of lift of the principal $\GL_r(\R)$-bundle of frames of $W$ to a principal $\Spin_r$-bundle, where $r$ is the rank of $W$. This data is equivalent to a trivialization of the Stiefel-Whitney classes $w_1(W)$ and $w_2(W)$~\cite[\S 26.5]{BH59}; i.e.\ data of nullhomotopies of the maps $Y\to K(\Z/2, 1)$ and $Y\to K(\Z/2, 2)$ representing $w_1(W)$, resp.\ $w_2(W)$.
\begin{defn}[{\cite[\S 4.1]{HKT20}}]
\label{what_is_twisted_spin}
Let $V\to X$ be a virtual vector bundle. An \term{$(X, V)$-twisted spin structure} on a virtual vector bundle $W\to Y$ is data of a map $f\colon Y\to X$ and a spin structure on $W\oplus f^*(V)$.
\end{defn}
This notion encompasses many commonly considered variations of spin structure.\footnote{However, see Stolz~\cite[\S 2.6]{Sto98} for a different notion of twisted spin structure and~\cite[\S 3.1]{DY23} for examples showing that Stolz' definition is strictly more general than~\cref{what_is_twisted_spin}.}
\begin{exm}
\label{spinc_tw_spin}
A \spinc structure on a virtual vector bundle $W\to Y$ is a reduction of the structure group of $W$ to the group~\cite[\S 3]{ABS64}
\begin{equation}\label{spinc_defn}
    \Spin^c_n\coloneqq \Spin_n\times_{\set{\pm 1}} \U_1,
\end{equation}
where the map to $\O_n$ is the composition
\begin{equation}
    \Spin_n^c\xrightarrow{\mathrm{proj}_1} \SO_n\hookrightarrow \O_n.
\end{equation}
This amounts to the data of a trivialization of $w_1(W)$ and a class $c\in H^2(Y;\Z)$ and an identification of $c\bmod 2 = w_2(W)$ (i.e.\ a trivialization of $c\bmod 2 + w_2(W)$). As $B\U_1$ is a $K(\Z, 2)$, there is a complex line bundle $L\to Y$ with $c_1(L) = c$, and $L$ is unique up to isomorphism.

The condition ``$c_1(L)\bmod 2 = w_2(W)$'' is equivalent to ``$W\oplus L$ is spin'': the Whitney sum formula shows $w_2(W\oplus L) = w_2(W)\oplus w_2(L)$, because $w_1(W)$ and $w_1(L)$ both vanish. Then, $w_2(V) = c_1(V)\bmod 2$ for any complex vector bundle $V$.

Finally, since all complex line bundles are pullbacks of the tautological bundle $V_{\U_1}\to B\U_1$ in a unique way up to isomorphism, the data of a \spinc structure on $W$ is equivalent to a map $f\colon Y\to B\U_1$ and a spin structure on $W\oplus f^*(V_{\U_1})$. That is, \spinc structures are $(B\U_1, V_{\U_1})$-twisted spin structures.
\end{exm}
\begin{exm}\label{pin_exm}
The same argument as in \cref{spinc_tw_spin} identifies several more kinds of twisted spin structures.
The \pinp and \pinm groups are defined as central extensions
\begin{equation}
    \shortexact*{\set{\pm 1}}{\Pin_n^\pm}{\O_n}.
\end{equation}
Central extensions of this form are classified by $H^2(B\O_n; \set{\pm 1})$; $\Pin_n^+$ is the extension corresponding to the class $w_2$, and $\Pin_n^-$ corresponds to $w_2 + w_1^2$.

Standard obstruction theory then implies a \pinp structure on a vector bundle $W\to Y$ is equivalent to a trivialization of $w_2(W)$, while a \pinm structure on $W$ is equivalent to a trivialization of $w_2(W) + w_1(W)^2$. A similar characteristic-class argument as in \cref{spinc_tw_spin} shows that \pinp structures are equivalent to $(B\Z/2, -\sigma)$-twisted spin structures, where $\sigma\to B\Z/2$ is the tautological bundle; similarly, \pinm structures are equivalent to $(B\Z/2, \sigma)$-twisted spin structures.

Campbell~\cite[\S 7.8]{Cam17} proves a related statement for $2\sigma$: $(B\Z/2, 2\sigma)$-twisted spin structures are equivalent to $G$-structures for $G = \Spin\times_{\set{\pm 1}}\Z/4$.
\end{exm}
\begin{exm}\label{pinc}
If one imitates the definition of $\Spin_n^c$ from~\eqref{spinc_defn} using the pin groups, the resulting group and its map down to $\O_n$ is the same whether one begins with $\Pin_n^+$ or $\Pin_n^-$. Thus using either, the group $\Pin_n^c$ is defined to be~\cite[Corollary 3.19]{ABS64}
\begin{equation}
    \Pin_n^c \coloneqq \Pin_n^\pm\times_{\set{\pm 1}} \U_1.
\end{equation}
The map to $\O_n$ is analogous to that for $\Spin_n^c$, and 
a \pinc structure on $W\to Y$ is the data of a class $c\in H^2(Y;\Z)$ with $c\bmod 2 = w_2(W)$; i.e.\ the same as a \spinc structure with no condition on $w_1$. This is equivalent to a $(B\Z/2\times B\U_1, \sigma \oplus V_{\U_1})$-twisted spin structure.
\end{exm}
\begin{exm}\label{spinh_tw_spin}
A quaternionically-minded reader might expect analogues of \cref{spinc_tw_spin,pinc} with $\SU_2$ in place of $\U_1$. Indeed, there are groups
\begin{subequations}
\begin{align}
    \Spin_n^h &\coloneqq \Spin_n\times_{\set{\pm 1}} \SU_2\\
    \Pin_n^{h\pm} &\coloneqq \Pin_n^\pm \times_{\set{\pm 1}} \SU_2.
\end{align}
\end{subequations}
%
Unlike for $\Pin_n^c$, $\Pin_n^{h+}$ and $\Pin_n^{h-}$ do not define equivalent tangential structures. Freed-Hopkins~\cite[(10.20)]{FH21} show that \spinh structures are equivalent to $(B\SO_3, V_{\SO_3})$-twisted spin structures and pin\textsuperscript{$h\pm$} structures are equivalent to $(B\O_3, \pm V_{\O_3})$-twisted spin structures, where for $G = \SO_3$ or $\O_3$, the bundle $V_G\to BG$ is the
tautological bundle.
\end{exm}
%
We will study bordism groups of manifolds with $(X, V)$-twisted spin structures. To do so, we express these twisted spin structures as untwisted tangential structures.
\begin{lem}[Shearing]\label{shearing}
Given $V\to X$ as above, $(X, V)$-twisted spin structures on a vector bundle $W\to Y$ are in natural bijection with homotopy classes of lifts in the diagram
\begin{equation}\begin{tikzcd}
	& {B\Spin\times X} \\
	Y & B\O.
	\arrow["{V_{\Spin} - V + \mathrm{rank}(V)}", from=1-2, to=2-2]
	\arrow["W"', from=2-1, to=2-2]
	\arrow[dashed, from=2-1, to=1-2]
\end{tikzcd}\end{equation}
Here $V_{\Spin}\to B\Spin$ is the tautological virtual vector bundle.
\end{lem}
The rank term appears so that the entire virtual bundle has rank zero.
\begin{cor}\label{shearing_spectra}
The Thom spectrum whose homotopy groups correspond under the Pontrjagin-Thom theorem to the bordism groups of manifolds with $(X, V)$-twisted structures on the \emph{tangent} bundle is homotopy equivalent as $MT\Spin$-modules to $MT\Spin\wedge X^{V - \mathrm{rank}(V)}$.
\end{cor}
Here the $-V$ becomes a $+V$ because we want tangential bordism, not normal bordism. 
For the same reason, we use the Madsen-Tillmann spectrum $MTH$, which is the Thom spectrum whose homotopy groups are the bordism groups of manifolds with $H$-structures on their stable tangent bundles. In homotopy theory, one more often encounters $MH$, which is the Thom spectrum for manifolds with $H$-structures on their stable \emph{normal} bundles. 
For $H = \Spin$, $\Spin^c$, and $\Spin^h$, there is a canonical homotopy equivalence $MTH\simeq MH$, but this is not true for all $H$: for example, $MT\Pin^{h\pm}\simeq M\Pin^{h\mp}$.


\subsection{The Smith long exact sequence}\label{s:LES}

Next, we introduce our main tool: the Smith long exact sequence. We detailed this long exact sequence in~\cite{DDKLPTT24} as part of a general framework for studying Smith homomorphisms. We begin by defining \textit{Smith maps}: maps of Madsen-Tillman spectra induced by the inclusion of the zero section of a vector bundle. 

Let $V\to X$ be as in the previous section, and let $W\to X$ be a real vector bundle. The inclusion of the zero section $0\hookrightarrow W$ induces a map of Thom spaces $X_+\to \mathrm{Th}(X;W)$, 
and upgrading this map to incorporate twisting by $V$ yields the following.

\begin{defn}
Let $V\to X$ be a virtual vector bundle and $W\to X$ be a vector bundle. The \term{Smith map} associated to $X$, $V$, and $W$ is the map of Thom spectra
    \begin{equation}\label{Thom_spectra_Smith}
        \mathrm{sm}_W\colon X^V \to X^{V+W}
    \end{equation}
    induced by the map $v\mapsto (v, 0)\colon V\to V\oplus W$.
\end{defn}
\begin{rem}
    In the case that $V$ is a vector bundle, the formula $v\mapsto (v, 0)$ describes an actual function of spaces which descends to a map of Thom spaces
    \begin{equation}\label{SmithThomspace}
        \mathrm{sm}_W\colon \mathrm{Th}(X;V) \to \mathrm{Th}(X;V\oplus W).
    \end{equation}
    Then~\eqref{Thom_spectra_Smith} is $\Sigma^\infty$ applied to~\eqref{SmithThomspace}.

    For more general $V$, one makes sense of \eqref{Thom_spectra_Smith} as follows: if $X$ is homotopy equivalent to a finite-dimensional CW complex, we may replace $V$ with an actual vector bundle up to a trivial summand, which only (de)suspends the map of Thom spectra. In general, one takes a colimit over $n$-skeleta. See~\cite[Definition 3.13]{DDKLPTT24} for more details.
\end{rem}
 We will consider the maps induced by Smith maps on spin bordism. There is a similar story for other generalized (co)homology theories; see~\cite[\S 7]{DDKLPTT24} for more examples.
 

\begin{defn}
    With $X$, $V$, and $W$ as above, let $r_V\coloneqq \mathrm{rank}(V)$ and $r_W\coloneqq\mathrm{rank}(W)$.
    The \textit{Smith homomorphism} associated to $X$, $V$, and $W$ is the homomorphism
    \begin{equation}\label{spin_Smith}
        \mathrm{sm}_W\colon \Omega^{\Spin}_n(X^{V-r_V}) \longrightarrow \Omega^{\Spin}_{n-r_W}(X^{V+W-r_V-r_W})
    \end{equation}
    induced by applying $\Omega_n^\Spin$ to~\eqref{Thom_spectra_Smith}. 
\end{defn}
We describe the homomorphism~\eqref{spin_Smith} on the level of manifolds in~\ref{geom_Smith}.
%

\begin{exm}\label{pinm_smith_iso}
    Let $X=B\Z/2$, $V=0$, and $W=\sigma$, the tautological line bundle. Including the zero section into $\sigma$ defines a map of spectra
    \begin{equation}
        \mathrm{sm}_\sigma\colon B\Z/2_+ \to (B\Z/2)^\sigma,
    \end{equation}
    and taking spin bordism gives the Smith homomorphism
    \begin{equation}
        \mathrm{sm}_\sigma\colon \Omega^\Spin_n(B\Z/2) \to \Omega^\Spin_{n-1}((B\Z/2)^{\sigma-1}).
    \end{equation}
    Using \cref{pin_exm,shearing}, we may recognize this as a map
    \begin{equation}
    \label{first_smith_isom}
        \mathrm{sm}_\sigma\colon \Omega^{\Spin}_n(B\Z/2) \to \Omega^{\Pinm}_{n-1}
    \end{equation}
    between the bordism groups of spin manifolds equipped with a principal $\Z/2$-bundle and \pinm manifolds. Restricted to reduced spin bordism, \eqref{first_smith_isom} is an isomorphism for all $n$~\cite{ABP69}; we will later establish this as a consequence of \cref{real_Smith_isom}.
\end{exm}

\begin{exm}\label{not_an_isom}
    Taking $X$ and $W$ as above but starting with $V=\sigma$, we obtain a map of spectra
    \begin{equation}
        \mathrm{sm}_\sigma\colon (B\Z/2)^{\sigma} \to (B\Z/2)^{2\sigma},
    \end{equation}
    which on spin bordism gives a Smith homomorphism
    \begin{equation}
        \mathrm{sm}_\sigma\colon \Omega^\Spin_n((B\Z/2)^{\sigma-1}) \to \Omega^\Spin_{n-1}((B\Z/2)^{2\sigma-2}).
    \end{equation}
    \Cref{pin_exm,shearing} allow us to rewrite this as a map
    \begin{equation}
        \mathrm{sm}_\sigma\colon \Omega^{\Pinm}_n \to \Omega^{\Spin\times_{\set{\pm 1}}\Z/4}_{n-1}.
    \end{equation}
    This map is generally not an isomorphism. For example, when $n = 2$, $\Omega_2^{\Pin^-}\cong\Z/8$~\cite[Theorem 5.1]{ABP69}\footnote{Within the statement of~\cite[Theorem 5.1]{ABP69}, the piece relevant for $\Omega_2^{\Pin^-}$ is ``The contribution to $\Omega_*^{\Pin}$ of terms $\pi_*(\mathit{RP}^\infty\wedge \mathbf BO\ang{8n})$ is as follows\dots{} ${Z_2}^{4k+3}$ in dim $8n+8k+2$, $k\ge 0$.'' For us $n = k = 0$. There is a typo: ${Z_2}^{4k+3}$ should be $Z_{2^{4k+3}}$. See Giambalvo~\cite[Theorem 3.4(b)]{Gia73} and Kirby-Taylor~\cite[Lemma 3.6]{KT90} for additional calculations of $\Omega_2^{\Pin^-}$.}
    and $\Omega_1^{\Spin\times_{\set{\pm 1}}\Z/4}\cong\Z/4$~\cite[Theorem 7.9]{Cam17}. Unlike in \cref{pinm_smith_iso}, we cannot solve this problem by discarding a basepoint.
\end{exm}

\begin{exm}\label{pinc_smith_iso}
    Next, we again take $X=B\Z/2$, $V=0$, and $W=\sigma$, but instead take \spinc bordism. In other words, we smash with $MT\Spinc$ instead of $MT\Spin$ and then take homotopy groups.
    We get
    \begin{equation}\label{spinc_pinc_interm}
        \sm_\sigma^c \colon \Omega_n^{\Spin^c}(B\Z/2) \to \Omega^{\Spin^c}_{n-1}((B\Z/2)^{\sigma-1}).
    \end{equation}
    Using \cref{spinc_tw_spin,pinc,shearing}, we can rewrite the codomain:
    \begin{equation}
        \Omega_{n-1}^{\Spin^c}((B\Z/2)^{\sigma-1}) \underset{\eqref{spinc_tw_spin}}{\cong} \Omega_{n-1}^{\Spin}((B\Z/2)^{\sigma-1}\wedge (B\U_1)^{V_{\U_1}-2}) \underset{\eqref{pinc}}{\cong} \Omega_{n-1}^{\Pin^c},
    \end{equation}
    allowing us to rephrase~\eqref{spinc_pinc_interm} as a map
    \begin{equation}\label{spinc_smith_map}
        \sm^c_\sigma\colon \Omega_n^{\Spinc}(B\Z/2) \to \Omega^\Pinc_{n-1}.
    \end{equation}
    Like in \cref{pinm_smith_iso}, when restricted to reduced \spinc bordism of $B\Z/2$,~\eqref{spinc_smith_map} is an isomorphism for all $n$. This is a theorem of Bahri-Gilkey~\cite{BG87}; we will prove it in \cref{real_Smith_isom}.

    We could equivalently describe this example using $X=B\Z/2\times B\U_1$, $V=0$, and $W=\sigma\oplus V_{\U_1}$ and taking spin bordism, applying \cref{pinc}.
\end{exm}

Proving the quaternionic analog of \Cref{pinm_smith_iso,pinc_smith_iso} is our objective in the next section.

Before then, we shall extend Smith homomorphisms to a long exact sequence, toward our second application. 
To do so, we identify the fiber of \cref{Thom_spectra_Smith}. We write $S(W)$ for the sphere bundle of a vector bundle $W\to X$.

The following theorem is attributed to James (see, e.g.,~\cite[Remark 3.14]{KZ18}). See for example~\cite[Theorem 5.1]{DDKLPTT24} for a proof.\footnote{This theorem can also be deduced from a theorem of Conner-Floyd~\cite[\S 16]{CF66}; see~\cite[\S 5.1]{DDKLPTT24} for more information.}
\begin{thm}
    \label{the_cofiber_sequence}
    Let $V$ be a virtual bundle and let $W$ be real vector bundle over $X$. Write $p\colon S(W)\to X$ for the projection map.
    Then there is a fiber sequence in spectra:
    \begin{equation}
    \label{the_cof_seq}
        S(W)^{p^*V} \to X^V \to X^{V \oplus W}.
    \end{equation}
\end{thm}

\begin{exm}\label{real_Smith_isom}
    Return to the setup with $X=B\Z/2$, $V=0$, and $W=\sigma$ from \cref{pinm_smith_iso}. Since the sphere bundle of $\sigma\to B\Z/2$ gives the universal fibration $E\Z/2\to B\Z/2$, we see that $S(\sigma)$ is contractible. Since $V=0$, we simply we get a fiber sequence
    \begin{equation}\label{isom_fib_seq}
        \mathbb{S} \to B\Z/2_+ \xrightarrow{\sm_\sigma} (B\Z/2)^\sigma.
    \end{equation}
    \begin{lem}\label{RP_split}
    The sequence~\eqref{isom_fib_seq} splits. Explicitly, the crush map $c\colon (B\Z/2)_+\to\mathbb S$ is a section and the restriction of $\sm_\sigma$ to the fiber of $c$ is a homotopy equivalence $\widetilde{\sm}_\sigma\colon \Sigma^\infty B\Z/2\simeq (B\Z/2)^\sigma$.
    \end{lem}
    This is a standard fact: see, e.g.~\cite[Lemma 2.6.5]{Koc96}. In a sense, its proof is trivial: the crush map always splits $\mathbb S$ off of $X_+$ for any space $X$, and the rest 
    follows from that and general properties of fiber sequences of spectra. The heart of 
    the lemma
    is that the inclusion of a basepoint in $B\Z/2$ suspends to participate in a Smith fiber sequence, which is more nontrivial.\footnote{Another way to approach \cref{RP_split} is to directly observe that the Thom space of $\sigma\to\RP^n$ is homeomorphic to $\RP^{n+1}$ and that the zero section inside the Thom space can be homotoped into the standard inclusion $\RP^n\inj\RP^{n+1}$ coming from the equatorial $S^n\inj S^{n+1}$. Then check compatibility as $n\to\infty$ and conclude.}
 %
After smashing with $MT\Spin$, $MT\Spin^c$, or $MT\Spin^h$ and invoking the Pontrjagin-Thom
theorem,~\eqref{isom_fib_seq} produces the Smith isomorphisms of \cref{pinm_smith_iso,pinc_smith_iso,main_thm}.\footnote{The Smith homomorphism interpretation of this equivalence is well-known, but
we are not sure who was the first to discuss it in general: see~\cite[\S 7.1]{DDKLPTT24} and the references therein.}
\end{exm}

There are many examples of Smith maps with more interesting fibers, including the Smith map \eqref{rational_Smith_iso_spectra} that we discuss in \cref{Question10.3section}.


\begin{cor}
    \label{bordism_LES_cor}
    Taking the spin bordism of \eqref{the_cof_seq} yields a homology long exact sequence:
    \begin{equation}
    \label{smith_the_LES}
        \cdots \longrightarrow \Omega^{\Spin}_n(S(W)^{p^*V}) \xrightarrow{~p_*} \Omega^{\Spin}_n(X^V) \xrightarrow{\mathrm{sm}_W} \Omega^{\Spin}_{n-r_W}(X^{V + W -r_W}) \xrightarrow{~\partial~} \Omega^{\Spin}_{n-1}(S(W)^{p^*V}) \longrightarrow \cdots
    \end{equation}
\end{cor}

The central map is the Smith homomorphism, and the other two are the pullback and the connecting homomorphism. This long exact sequence is remarkably useful for bordism computations, specifically for resolving extension problems that arise in spectral sequence calculations. Moreover, we understand this sequence on the level of manifolds:

\begin{enumerate}[label=\textrm{(LES-\arabic*)}]
    \item\label{LES_map_p} Let $[M,h]$ be the bordism class of an $n$-manifold $M$ equipped with a map $h\colon M\to S(W)$ such that $TM\oplus h^*p^*V$ is spin, so that $[M,h]\in \Omega_n^\Spin(S(W)^{p^*V})$. Its image under $p^*$ is the class $[M,M\xrightarrow{h}S(W)\xrightarrow{p}X]$ of the same underlying manifold equipped with a map to $X$ given by composing with the projection.

    \item\label{geom_Smith} Let $[M, f]\in \Omega^\Spin_n(X^V)$, so that $M$ is an $n$-manifold such that $TM\oplus f^*V$ is spin. Consider the pullback of $W$ to $M$. The intersection of the zero section of $W$ with a generic section is, by transversality, a submanifold $N$ of codimension $r_W$. The image of $M$ under the Smith homomorphism is the class $[N, N\hookrightarrow M\xrightarrow{f} X]$. In this setting, the normal bundle $\nu\to N$ of the embedding $i\colon N\hookrightarrow M$ is isomorphic to $W|_N$, so $TN \oplus (i\circ f)^*W \oplus (i\circ f)^*V\cong i^*(TM\oplus f^*V)$ is spin, and therefore $N$ has a $(X, V\oplus W)$-twisted spin structure.

    \item\label{final_LES} Let $[N,g]$ be a class in $\Omega_{n}^\Spin(X^{V+W-r_W})$, so $g\colon N \to X$ is such that $TN+g^*(V+W)$ is spin. The image of $[N,g]$ under the connecting homomorphism is the class $[S(W)|_N, S(W)|_N\hookrightarrow S(W)]\in \Omega_{n-1}^\Spin(S(W))$ given by restricting the sphere bundle of $W$ to $N$.
\end{enumerate}
For a justification of these descriptions, see \cite[Appendix A]{DDKLPTT24}.


\section{A \texorpdfstring{$\Pin^{h-}$}{Pinh minus} Smith Isomorphism}\label{smith_iso_section}
In this section we answer \cite[Question 10.8]{BM23} (\cref{spinh_smith_qn} in this article): is there a Smith isomorphism for \pinhm bordism, generalizing \cref{pinm_smith_iso,pinc_smith_iso}?

%
\begin{thm}
\label{main_thm}
For all $n$, there is an isomorphism $\mathrm{sm}^h_\sigma\colon\widetilde\Omega_{n+1}^{\Spin^h}(\RP^\infty)
\xrightarrow{\cong} \Omega_n^{\Pin^{h-}}$ given by sending a pair $(M,f)$ of a \spinh manifold $M$ with a map $f\colon M\to\RP^\infty$
(which may without loss of generality be assumed to be transverse to $\RP^{\infty-1}\subset\RP^\infty$) to the
\pinhm manifold $f^{-1}(\RP^{\infty-1})$.
\end{thm}
The idea of 
our proof
is this: using \cref{spinh_tw_spin,shearing},
\pinhm bordism is isomorphic to the spin bordism of the Thom
spectrum $(B\O_3)^{3-V_{\O_3}}$, and \spinh bordism is isomorphic to the spin bordism of $(B\SO_3)^{3-V_{\SO_3}}$, where again $V_G\to BG$ denotes a tautological bundle. The
isomorphism $\O_3\cong\SO_3\times\Z/2$ allows one to factor $(B\O_3)^{3-V_{\O_3}}$ as a smash product of $(B\SO_3)^{3-V_{\SO_3}}$ and a piece that corresponds to $\RP^\infty$, leading to the isomorphism in the theorem statement.

Now we give the details, beginning with some lemmas.
Recall that $H^*(B\O_3;\Z/2)\cong\Z/2[w_1,w_2,w_3]$ with $\abs{w_i} = i$~\cite{Bor53}, $H^*(B\Z/2;\Z/2)\cong\Z/2[a]$ with $\abs{a} = 1$, and $H^*(B\SO_3;\Z/2)\cong\Z/2[\overline w_2,\overline w_3]$ with $\abs{\overline w_i} = i$~\cite[Proposition 8.1]{Bor53}. (The classes $\overline w_i$ are the usual Stiefel-Whitney classes, but we write $\overline w$ so that the classes on $B\O_3$
and $B\SO_3$ have different names.)
\begin{lem}\label{pullchar}
Write $\vp\colon \SO_3\times\Z/2 \overset\cong\to \O_3$ for the isomorphism. Then the map $\vp^*\colon H^*(B\O_3;\Z/2)\to H^*(B\Z/2\times B\SO_3;\Z/2)$ on cohomology is such that $\vp^*(w_1) = a$ and $\vp^*(w_2) = a^2 + \overline w_2$.
\end{lem}


\begin{proof}
Since $\vp$ is an isomorphism, so is $\vp^*$. Therefore, since $w_1\ne 0$, $\vp^*(w_1)$ must also be nonzero. Since
$H^1(B\Z/2;\Z/2)\cong\Z/2\cdot a$ and $H^1(B\SO_3;\Z/2) \cong 0$, the Künneth formula tells us that the only nonzero class, which
must be $\vp^*(w_1)$, is $a$.

To match $w_2$, we have three nonzero classes: $\overline w_2$, $a^2$, and $\overline w_2 + a^2$. To tell them apart,
first consider the map $i_1\colon B\SO_3\to B\O_3$ induced by the inclusion $\SO_3\inj\O_3$; this factors through
$\vp$ and $i_1^*(w_2) = \overline w_2$ by definition, so $\vp^*(w_2)$ must be either $\overline w_2$ or $\overline
w_2 + a^2$. Likewise, take the map $i_2\colon B\Z/2\to B\O_3$ induced by $\vp$; since this map is defined by
sending $1\in\Z/2$ to an inversion in $\O_3$, $\Z/2$ acts on the pullback representation $i_2^*V_{\O_3}$ as $\set{\pm
1}$, i.e.\ as the representation $3\sigma$. Thus $i_2^*(w_2) = w_2(3\sigma)$,
which by the Whitney sum formula is $w_1(\sigma)^2$, i.e.\ $a^2$. Thus $\vp^*(w_2)$ must have an $a^2$ term, and we
already saw it must have an $\overline w_2$ term, so $\vp^*(w_2) = \overline w_2 + a^2$.
\end{proof}
If $V_1\to X_1$ and $V_2\to X_2$ are virtual vector bundles, then there is a homotopy
equivalence $(X_1\times X_2)^{V_1\boxplus V_2}\simeq (X_1)^{V_1}\wedge (X_2)^{V_2}$. Since $B\O_3$ splits as a
direct product, one might hope that $3-V_{\O_3}\to B\O_3$ is an external direct sum, leading to a splitting of $(B\O_3)^{3-V_{\O_3}}$.
This is not true, but we will be able to replace $3-V_{\O_3}$ with a different vector bundle that is
an external direct sum using the following lemma.
\begin{lem}[{Relative Thom isomorphism, c.f.~\cite[Theorem 1.39]{Deb21} or~\cite{DY23}}]
\label{relative_Thom}
Let $V_1,V_2\to X$ be rank-zero vector bundles. A spin structure on $V_2$ determines a homotopy equivalence of
$MT\Spin$-module spectra $MT\Spin\wedge X^{V_1}\xrightarrow{\simeq} MT\Spin\wedge X^{V_1 + V_2}$.
\end{lem}
We will replace $\pm(V_{\O_3}-3)$ with $\pm((3\sigma-3) \boxplus (V_{\SO_3}-3))$, so we must check the hypothesis of
\cref{relative_Thom}.
\begin{lem}
\label{difference_spin}
The virtual vector bundles $\vp^*(\pm(V_{\O_3}-3)) - \pm((3\sigma-3) \boxplus (V_{\SO_3}-3))$ are spin.
\end{lem}
\begin{proof}
Directly compute with the Whitney sum formulas that $w_1(V_1\oplus V_2) = w_1(V_1) + w_1(V_2)$ and $w_2(V_1\oplus V_2) =
w_2(V_1) + w_1(V_1)w_1(V_2) + w_2(V_2)$. 
For any vector bundle $E$, setting $V_1 = E$ and $V_2 = -E$ (so that $V_1\oplus
V_2 = 0$) gives that $w_1(-E) = w_1(E)$ and $w_2(-E) = w_2(E) + w_1(E)^2$. Stability of the Stiefel-Whitney classes
implies we may add or subtract trivial bundles without affecting their characteristic classes.

Thus, for $E_+\coloneqq \vp^*(V_{\O_3}-3) - ((3\sigma-3)\boxplus V_{\SO_3}-3)$, we have using \cref{pullchar} that
\begin{equation}
\begin{aligned}
	w_1(E_+) &= \textcolor{BrickRed}{w_1(\vp^*(V_{\O_3}-3))} + w_1(-((3\sigma-3)\boxplus (V_{\SO_3}-3)))\\
	&= \textcolor{BrickRed}{\vp^*(w_1(V_{\O_3}))} + \textcolor{Green}{w_1(-3\sigma)} + \textcolor{MidnightBlue}{w_1(-V_{\SO_3})}\\
	&= \textcolor{BrickRed}{\vp^*(w_1)} + \textcolor{Green}{w_1(3\sigma)} + \textcolor{MidnightBlue}{w_1(V_{\SO_3})}\\
	&= \textcolor{BrickRed}{a} + \textcolor{Green}{a} + \textcolor{MidnightBlue}{0} = 0,
\end{aligned}
\end{equation}
and
\begin{equation}
\begin{aligned}
	w_2(E_+) &= \textcolor{BrickRed}{w_2(\vp^*(V_{\O_3}-3))} +
		\textcolor{RedOrange}{w_1(\vp^*(V_{\O_3}-3))w_1(-(3(\sigma-3)\boxplus (V_{\SO_3}-3)))} + w_2(-((3(\sigma-3)\boxplus
		(V_{\SO_3}-3)))\\
	&= \textcolor{BrickRed}{\vp^*(w_2(V_{\O_3}))} + \textcolor{RedOrange}{\vp^*(w_1(V_{\O_3}))(w_1(-3\sigma) + w_1(-V_{\SO_3}))} +
		\textcolor{Green}{w_2(-3\sigma)} + \textcolor{MidnightBlue}{w_1(-3\sigma)w_1(-V_{\SO_3})} +
		\textcolor{Fuchsia}{w_2(-V_{\SO_3})}\\
	&= \textcolor{BrickRed}{\overline w_2 + a^2} + \textcolor{RedOrange}{a^2} + \textcolor{Green}{w_2(3\sigma) +
		w_1(3\sigma)^2} + \textcolor{MidnightBlue}{w_1(3\sigma)w_1(W)} + \textcolor{Fuchsia}{w_2(V_{\SO_3}) + w_1(V_{\SO_3})^2}\\
	&= \textcolor{BrickRed}{\overline w_2 + a^2} + \textcolor{RedOrange}{a^2} + \textcolor{Green}{a^2 + a^2} +
		\textcolor{MidnightBlue}{0} + \textcolor{Fuchsia}{\overline w_2 + 0} = 0.
\end{aligned}
\end{equation}
Since the first and second Stiefel-Whitney classes of $E_+$ vanish, $E_+$ is spin. This also implies $w_1$ and $w_2$ of $-E_+$ vanish, so we have proven the claim for both bundles in the lemma statement.
\end{proof}
\begin{cor}
\label{factor_spectra}
There are equivalences of spectra (in fact, of $MT\Spin$-module spectra)
\begin{equation}
	MT\Spin\wedge (B\SO_3)^{V_{\SO_3}-3}\wedge (B\Z/2)^{\pm(3\sigma-3)} \simeq MT\Spin\wedge (B\O_3)^{\pm(V_{\O_3}-3)}.
\end{equation}
\end{cor}
\begin{proof}
We prove the $+$ case; the $-$ case is analogous. \Cref{relative_Thom} tells us that, since $\vp^*(V_{\O_3}-3) -
((3\sigma-3) \boxplus(V_{\SO_3}-3))$ is spin and rank-zero, there are equivalences of $MT\Spin$-module spectra
\begin{equation}
\label{thom_split}
\begin{aligned}
	MT\Spin\wedge (B\O_3)^{V_{\O_3}-3} &\simeq MT\Spin\wedge (B\Z/2\times B\SO_3)^{\vp^*V_{\O_3} - 3}\\
    &\simeq MT\Spin\wedge
	(B\Z/2\times B\SO_3)^{(3\sigma-3)\boxplus (V_{\SO_3}-3)}.
\end{aligned}
\end{equation}
As we noted above, the Thom spectrum functor sends external direct sums to smash products,
so the Thom spectrum
in~\eqref{thom_split} factors as $MT\Spin\wedge (B\Z/2)^{3\sigma-3}\wedge (B\SO_3)^{V_{\SO_3}-3}$, as we wanted to prove.
\end{proof}
Now we are ready to prove the main theorem of this section.
\begin{proof}[Proof of \cref{main_thm}]
\Cref{spinh_tw_spin,shearing} combine to produce homotopy equivalences
\begin{subequations}
\label{twisted_str}
\begin{align}
	MT\Spin^h &\simeq MT\Spin\wedge (B\SO_3)^{V_{\SO_3}-3}\\
	MT\Pin^{h\pm} &\simeq MT\Spin\wedge (B\O_3)^{\pm(V_{\O_3}-3)},
\end{align}
\end{subequations}
which are originally due to Freed-Hopkins~\cite[(10.2)]{FH21}.
Combining~\eqref{twisted_str} with \cref{factor_spectra}, we have produced equivalences
\begin{equation}
	MT\Pin^{h\pm} \simeq MT\Spin^h \wedge (B\Z/2)^{\pm(3\sigma-3)}.
\end{equation}

One can check using the Whitney sum formula that the bundle $4\sigma\to B\Z/2$ has a spin structure. Thus we may
once again invoke \cref{relative_Thom} to obtain an equivalence $MT\Spin\wedge (B\Z/2)^{-(3\sigma-3)}\simeq
MT\Spin\wedge (B\Z/2)^{\sigma-1}$: the difference between the two vector bundles is $4\sigma-4$, which is spin, so
adding $4\sigma-4$ to $-(3\sigma-3)$ does not change the homotopy type.

The only remaining task is to get from $(B\Z/2)^{\sigma-1}$ to $\RP^\infty$ and interpret the resulting equivalence as a Smith isomorphism. This is done in \cref{real_Smith_isom}.
\end{proof}
\begin{rem}
Buchanan-McKean's original question asked about a Smith isomorphism between \pinhm bordism and the \spinh bordism of $\HP^\infty$. These bordism groups are not isomorphic: to see this, run the Atiyah-Hirzebruch spectral sequence
\begin{equation}
	E^2_{p,q} = \widetilde H_p(\HP^\infty; \Omega_q^{\Spin^h}(\pt)\otimes \Q) \Longrightarrow \widetilde
	\Omega_{p+q}^{\Spin^h}(\HP^\infty)\otimes\Q.
\end{equation}
All Atiyah-Hirzebruch spectral sequences with $\Q$ coefficients collapse, so $E^2_{4,0}\cong \widetilde
H_4(\HP^\infty; \Q)\cong\Q$~\cite[\S 15.5]{BH58} implies $\widetilde\Omega_4^{\Spin^h}(\HP^\infty)\otimes\Q\ne 0$, but
$\Omega_3^{\Pin^{h-}}$ is zero~\cite[Theorem 9.97]{FH21}.

The \pinm and \pinc Smith isomorphisms of \cref{pinm_smith_iso,pinc_smith_iso} both use $B\Z/2\simeq\RP^\infty$, ultimately because $\Pin_n^-$, resp.\ $\Pin_n^c$ are extensions of $\Z/2$ by $\Spin_n$, resp.\ $\Spin_n^c$. That $\Pin_n^{h-}$ is also an extension of $\Z/2$, this time by $\Spin_n^h$, suggested to us that the Smith isomorphism should also use $B\Z/2$. Smith isomorphisms involving $B\SU_2\simeq \HP^\infty$ do exist~\cite[Example 7.42]{DDKLPTT24}, in a setting where one group is an extension of $\SU_2$ by another.
\end{rem}

\section{Rational Generators for \texorpdfstring{$\Spin^h$}{Spinh} Bordism from \texorpdfstring{$\Spin^c$}{Spinc} Bordism}\label{Question10.3section}

Next, we address the Question 10.3 asked by Buchanan and McKean in \cite{BM23} comparing \spinc and \spinh bordism in dimensions $0\bmod 4$.
\begin{thm}[{Buchanan-McKean~\cite[Corollary 8.6]{BM23}}]
\label{BM_same_rank}
For all $k\ge 0$, $\mathrm{rank}(\Omega_{4k}^{\Spin^c}) = \mathrm{rank}(\Omega_{4k}^{\Spin^h})$.
\end{thm}
\begin{ques}[{\cite[Question 10.3]{BM23}}]
\label{spinc_spinh_ranks}
Is there a geometric explanation for the equality of ranks in \cref{BM_same_rank} between degree-$4k$ \spinc and \spinh bordism?
Specifically, is there a procedure for producing generators for the free part of $\Omega_{4k}^{\Spin^h}$ from that of $\Omega_{4k}^{\Spin^c}$?
\end{ques}
We use the Smith long exact sequence to mostly answer this question: it provides a geometric explanation for the equality of ranks and allows one to produce \emph{rational} generators for \spinh bordism from generators of \spinc bordism. In the course of the proof, we will lift from $\Q$ to $\Z[1/2]$, but we will also see why it is hard to lift to a result over $\Z$.
\begin{constr}
The inclusion $\set{\pm 1}\hookrightarrow\SU_2$ used in the definition of $\Spin_n^h$ (\cref{spinh_tw_spin}) factors as the composition of the usual inclusion $\set{\pm 1}\hookrightarrow \U_1$ and the standard inclusion $\U_1\hookrightarrow \SU_2$. Taking the product with $\Spin_n$ and quotienting by the diagonal central $\set{\pm 1}$ subgroup, we obtain an inclusion $\iota\colon \Spin_n^c\hookrightarrow \Spin_n^h$ commuting with the structure maps to $\O_n$.

Given a vector bundle $V\to X$ with \spinc structure $\mathfrak s$, the \spinh structure $\iota(\mathfrak s)$ is called the \term{induced \spinh structure} of $\mathfrak s$.
\end{constr}
%
%
\begin{thm}\label{thm_answer_to_10.3}\hfill
\begin{enumerate}
    \item\label{10.3_Smith} Taking the induced \spinh structure of a \spinc structure defines a map of bordism groups $\Omega_n^{\Spin^c}\to\Omega_n^{\Spin^h}$ that participates in a Smith long exact sequence.
    \item\label{10.3_Z12} The induced map $\Omega_{4k}^{\Spin^c}\otimes\Z[1/2]\to\Omega_{4k}^{\Spin^h}\otimes\Z[1/2]$ is an isomorphism.
\end{enumerate}
%
\end{thm}
In particular, part~\eqref{10.3_Z12} follows from~\eqref{10.3_Smith} and a few computations in the literature. In light of the explicit interpretations of the maps in a Smith long exact sequence in~\ref{LES_map_p}--\ref{final_LES}, we believe \cref{thm_answer_to_10.3} provides a geometric answer to the first part of \cref{spinc_spinh_ranks}.

Consider the Smith map of spectra \eqref{Thom_spectra_Smith} for $X=B\SO_3$ and $V=W=V_{\SO_3}$, where $V_{\SO_3}$ is the tautological rank-three oriented bundle:
\begin{equation}\label{rational_Smith_iso_spectra}
    \mathrm{sm}_{V_{\SO_3}}\colon (B\SO_3)^{V_{\SO_3}}\to (B\SO_3)^{2V_{\SO_3}}
\end{equation}
By \cref{the_cofiber_sequence}, the fiber is given by the Thom spectrum over the sphere bundle: $S(V_{\SO_3})^{p^*V_{\SO_3}}$.
\begin{lem}\label{sphere_bundle_lemma}
For all $n\ge 1$, there is a homotopy equivalence $\varphi\colon S(V|_{\SO_n})\xrightarrow{\simeq} B\SO_{n-1}$, and $\varphi$ identifies the bundle map $p\colon S(V|_{\SO_n})\to B\SO_n$ with the map $B\SO_{n-1}\to B\SO_n$ induced by the standard inclusion $\SO_{n-1}\hookrightarrow\SO_n$, up to homotopy.
\end{lem}
This is well-known; see~\cite[Example 7.57]{DDKLPTT24} for a proof. Because of \cref{sphere_bundle_lemma}, we will also write $p$ for the map $B\U_1\to B\SO_3$ induced by the standard inclusion $\U_1\cong\SO_2\hookrightarrow\SO_3$.
%
Then, the pullback $p^*V_{\SO_3}$ to $B\SO_2$ is the rank-two tautological bundle over $B\SO_2$ plus a trivial real line bundle, and under the equivalence $\SO_2\cong \U_1$, the tautological rank-two oriented bundle $V_{\SO_2}\to B\SO_2$ is identified with the tautological complex line bundle $V_{\U_1}\to B\U_1$.\footnote{One way to see this is that these two vector bundles are induced from the defining representations $\SO_2\to\GL_2(\R)$, resp.\ $\U_1\to\GL_1(\C)$, and that the standard isomorphism $\C\cong\R^2$ induces an isomorphism of these two representations, hence also of their associated bundles.}
Overall, we have argued an equivalence
\begin{equation}
    S(V_{\SO_3})^{p^* V_{\SO_3}} \simeq (B\U_1)^{V_{\U_1}\oplus\,\underline\R}\simeq \Sigma (B\U_1)^{V_{\U_1}}.
\end{equation}
To study spin bordism, we smash the fiber sequence for $\sm_{V_{\SO_3}}$ with $MT\Spin$: 
\begin{equation}
    MT\Spin\wedge \Sigma (B\U_1)^{V_{\U_1}} \to MT\Spin \wedge (B\SO_3)^{V_{\SO_3}} \xrightarrow{\mathrm{sm}_{V_{\SO_3}}} MT\Spin\wedge (B\SO_3)^{2V_{\SO_3}}.
\end{equation}
Under shearing, this sequence becomes more familiar. Using \cref{spinc_tw_spin} and \cref{shearing_spectra}, we may recast the first spectrum as $\Sigma^3 MT\Spinc$. By \cref{spinh_tw_spin}, the second spectrum becomes $\Sigma^3 MT\Spinh$. Finally, the third spectrum represents $(B\SO_3,2V_{\SO_3})$-twisted spin bordism, but this twist is actually this is no twist at all: since $2V_{\SO_3}$ is spin, this spectrum reduces to $\Sigma^6 MT\Spin\wedge (B\SO_3)_+$ by \cref{relative_Thom}.

Altogether, after desuspending thrice, we have a fiber sequence of spectra\footnote{This fiber sequence and its corresponding Smith homomorphism also appears in~\cite[Example 7.45 and Appendix B]{DDKLPTT24}; it has the interesting property that the Smith homomorphism cannot be defined using ordinary cohomology: one must take the Poincaré dual submanifold in $\mathit{ko}$-cohomology or in spin cobordism.}
\begin{equation}
    MT\Spinc \overset{p}{\longrightarrow} MT\Spinh \xrightarrow{\mathrm{sm}_{V_{\SO_3}}} \Sigma^3 MT\Spin\wedge (B\SO_3)_+.
\end{equation}
The associated Smith long exact sequence is
\begin{equation}\label{spinc_h_smith}
    \dots \longrightarrow \Omega^{\Spinc}_n \overset{p_*}{\longrightarrow} \Omega^{\Spinh}_n \xrightarrow{\sm_{V_{\SO_3}}} \Omega_{n-3}^\Spin(B\SO_3) \overset\partial\longrightarrow
    \Omega^{\Spin^c}_{n-1}\longrightarrow\cdots
\end{equation}
\Cref{sphere_bundle_lemma} and~\ref{LES_map_p} imply that $p_*$ is the map taking the induced \spinh structure of a \spinc manifold, proving the first part of \cref{thm_answer_to_10.3}.

We are interested in~\eqref{spinc_h_smith} in degrees $i=4k$ after inverting $2$.
\begin{lem}\label{rational_where}\hfill
\begin{enumerate}
    \item\label{spinc_12} $\Omega_*^{\Spin^c}\otimes\Z[1/2]$ is concentrated in even degrees.
    \item\label{spinh_12} $\Omega_*^{\Spin^h}\otimes\Z[1/2]$ is concentrated in degrees $0\bmod 4$.
    \item\label{spinso3_12} $\Omega_*^\Spin(B\SO_3)\otimes\Z[1/2]$ is concentrated in degrees $0\bmod 4$.
\end{enumerate}
\end{lem}
\begin{proof}
\eqref{spinc_12} is in Stong~\cite[Chapter XI, p.\ 349]{Sto68}. For~\eqref{spinh_12}, use the equivalence $\Omega_*^{\Spin^h}\otimes\Z[1/2]\cong\Omega_*^\Spin\otimes H_*(B\SU_2;\Z[1/2])$~\cite[Remark A.2]{Hu23} together with the fact that both $\Omega_*^\Spin\otimes\Z[1/2]$ and $H_*(B\SU_2;\Z)$ are concentrated in degrees $0\bmod 4$ (\cite{anderson_structure_1967}, resp.\ \cite[\S 29]{Bor53a}); use the universal coefficient theorem to get to $H_*(B\SU_2;\Z[1/2])$ and thus to $\Omega_*^\Spin\otimes\Z[1/2]$.

For~\eqref{spinso3_12},
 %
use the Atiyah-Hirzebruch spectral sequence of signature
\begin{equation}
E^2_{p,q} = H_p(B\SO_3; \Omega_q^\Spin\otimes \Z[1/2]) \implies \Omega^{\Spin}_{p+q}(B\SO_3)\otimes \Z[1/2]
\end{equation}
    to compute. As noted above, spin bordism tensored with $\Z[1/2]$ is concentrated in degrees $0\bmod 4$, and $H_*(B\SO_3;\Z[1/2])\cong \Z[1/2, p_1]$ is concentrated in degrees divisible by $4$ as well~\cite[\S 29]{Bor53a}, so the spectral sequence collapses on $E_2$ and the result is also concentrated in degrees $0\bmod 4$.  
\end{proof}
\begin{thm}\label{generator_cor}
For all $k\ge 0$, the map
\begin{equation}
    p_*\colon\Omega_{4k}^\Spinc \longrightarrow \Omega_{4k}^\Spinh
\end{equation}
defined by taking the induced \spinh structure
is an isomorphism after tensoring with $\Z[1/2]$.
\end{thm}
\begin{proof}
We discussed the interpretation of $p_*$ as taking the induced \spinh structure right after~\eqref{spinc_h_smith}, so all that remains is the isomorphism away from $2$. Tensor the long exact sequence~\eqref{spinc_h_smith} with $\Z[1/2]$; since $\Z[1/2]$ is a flat $\Z$-module, the resulting sequence is still exact. Then plug in \cref{rational_where} and conclude.
%
\end{proof}
This finishes the proof of \cref{thm_answer_to_10.3}.



\newcommand{\etalchar}[1]{$^{#1}$}

\end{document}